\documentclass[]{paper}
\usepackage{latexsym}
\usepackage{epsfig}
\usepackage{amsmath}
\usepackage{amssymb}
\usepackage{cite}
\title{An Expandable Local and Parallel Two-Grid Finite Element Scheme}
\author{Yanren Hou\thanks{School of Mathematics and Statistics, Xi'an Jiaotong University,
        Xi'an, Shaanxi 710049, China. ({\tt yrhou@mail.xjtu.edu.cn})}, GuangZhi Du\thanks{School of Mathematics and Statistics, Xi'an Jiaotong University,
        Xi'an, Shaanxi 710049, China. }}
\date{}

\newtheorem{theorem}{\bf Theorem}[section]
\newtheorem{lemma}{\bf Lemma}[section]
\newtheorem{corollary}{\bf Corollary}[section]

\def\proof{\noindent{\bf Proof.}\hskip2mm}
\def\endproof{\hfill$\Box$}
\numberwithin{equation}{section}
\allowdisplaybreaks
\begin{document}
\maketitle
\begin{abstract}
An expandable local and parallel two-grid finite element scheme based on superposition principle for elliptic problems is proposed and analyzed in this paper by taking example of Poisson equation. Compared with the usual local and parallel finite element schemes, the scheme proposed in this paper can be easily implemented in a large parallel computer system that has a lot of CPUs. 
Convergence results base on $H^1$ and $L^2$ a priori error estimation of the scheme are obtained, which show that the scheme can reach the optimal convergence orders within $|\ln H|^2$ or $|\ln H|$ two-grid iterations if the coarse mesh size $H$ and the fine mesh size $h$ are properly configured in 2-D or 3-D case, respectively. Some numerical results are presented at the end of the paper to support our analysis.
\end{abstract}
\noindent{\bf Key Words} two-grid finite element method, domain decomposition method, local and parallel algorithm, error estimation

\noindent{\bf MSC2000} 65N15, 65N30, 65N55

\section{Introduction}

Two-grid or multi-grid finite element methods and domain decomposition methods are powerful tools for numerical simulation of solutions to PDEs with high resolution, which are otherwise inaccessible due to the limits in
computational resources. For examples, the domain decomposition schemes, nonlinear Galerkin schemes and two-grid/two-level post-processing schemes in \cite{AMMI,  BANK, CHAN, FOIAS, HACK, HOULI1, HOULI2, LIHOU, LIUHOU, XU} and the references therein. In the past decade, a local and parallel two-grid finite element method for elliptic boundary value problems was initially proposed in \cite{XU1} and was extended to nonlinear elliptic boundary value problems in \cite{XU2} and Stokes and Navier-Stokes equations in \cite{HEXU, HEXU1}.

Let us briefly recall the local and parallel two-grid finite element method in \cite{XU1} for the following simple Poisson equation with Dirichlet boundary condition defined in convex domain $\Omega\subset R^d$, $d=2,3$:
\begin{equation}\label{equ}
\left\{\begin{array}{rl}
-\Delta u=f, & \mbox{ in }\Omega,\\
u=0, & \mbox{ on }\partial\Omega,
\end{array}\right.
\end{equation}
whose weak formulation is: find $u\in H_0^1(\Omega)$ such that
\begin{equation}\label{wequ}
a(u,v)=(f,v),\quad\forall v\in H_0^1(\Omega).
\end{equation}

Suppose $T^H(\Omega)$ is a regular coarse mesh triangulation of $\Omega$ and $S^H(\Omega)\subset H^1(\Omega)$, $S_0^H(\Omega)=S^H(\Omega)\cap H_0^1(\Omega)$ are the corresponding finite element spaces defined on $T^H(\Omega)$ . Let us decompose the entire domain $\Omega$ into a series of disjoint subdomains, $\overline{\Omega}=\bigcup\limits_{j=1}^N \overline{D_j}$. For example, see Fig. \ref{FIGXU}.
\begin{center}
 \begin{figure}[ht]
  \centering
   \includegraphics[width=0.5\linewidth]{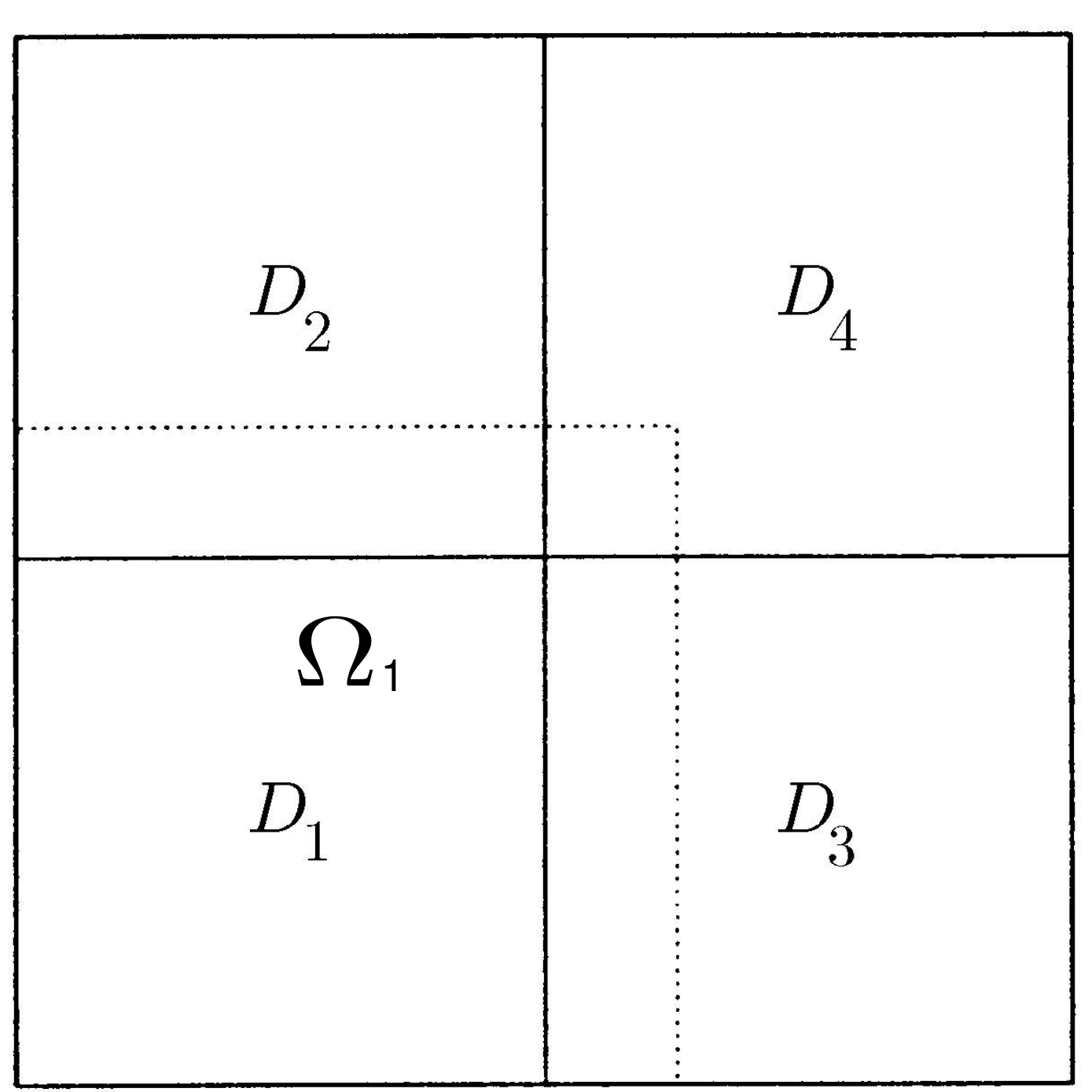}
   \caption{Decomposition of the domain $\Omega$}\label{FIGXU}
 \end{figure}
\end{center}
If the coarse mesh standard Galerkin approximation $u_H\in S_0^H(\Omega)$ is obtained, by expanding each subdomain $D_j$ to another subdomain $\Omega_j\subset\Omega$ and for a given fine mesh size $h<H$, one of the local and parallel two-grid schemes proposed in \cite{XU1} is: find $e_h^j\in S_0^h(\Omega_j)$ such that
\begin{equation}\label{XU1}
a(e_h^j,v)=(f,v)-a(u_H,v),\quad\forall v\in S_0^h(\Omega_j).
\end{equation}
And the final approximation $u^h$ is defined piecewisely by
$$
u^h=u_H+e_h^j,\quad\mbox{in}\quad D_j,\quad j=1,2,\cdots,N.
$$

Error estimations in \cite{XU1} shows that $u^h$ can reach the optimal convergence order in $H^1$ norm.
However, it is obvious that $u^h$ is in general discontinuous and its $L^2$ error bound does not in general have higher order than its $H^1$ error bound. To overcome this defect of the algorithm, the authors in \cite{XU1} modified the above scheme to ensure the continuity of $u^h$ in $\Omega$ and finally do a coarse grid correction to get the optimal error bound in $L^2$ norm. The most attractive feature of the algorithm is that the series of subproblems are independent once $u_H$ is known and therefore it is a highly parallelized algorithm. 

On the other hand, one can easily see from the content of \cite{XU1} that their error estimates heavily depend on the usage of the superapproximation property of finite element spaces. 
Thanks to \cite{NITSCHE}, we know that the usage of this property makes the error constant appeared in \cite{XU1} has the form $O(t^{-1})$, where $t=\mbox{dist}(\partial D_j\backslash\partial\Omega,\partial\Omega_j\backslash\partial\Omega)$. To guarantee the error orders obtained in \cite{XU1}, one should demand that
$t=O(1)$. That means the distance between the boundaries of a specific subdomain $D_j$ and its expansion $\Omega_j$ should be of constant order.  Therefore $\Omega_j$ could not be arbitrary small even when $\mbox{diam}(D_j)$ tends to zero. This will lead to a vast waste of parallel computing resources.

In this paper, we follow the basic idea presented in \cite{XU1} to construct another form of two-grid local and parallel scheme, in which the scale of each subproblem can be much smaller compared with that in \cite{XU1}. In fact, we will deal with the case of $\mbox{diam}(D_j)=O(H)$ and $t=O(H)$. We call the scheme an expandable local and parallel two-grid scheme because the scale of each subproblem can be arbitrary small as $H$ tends to zero and every two adjacent subproblems only have a small overlapping. Similarly, to get a better $L^2$ error bound, a coarse grid correction is done in each cycle of two-grid iteration.

Different from the previously mentioned local and parallel schemes, we use superposition principle to generate a series of local and independent subproblems and this will make the global approximation continuous in $\Omega$. Such kind of technique has been successfully used in \cite{LARSON, LARSON1}, in which adaptive variational multi-scale methods were constructed. In fact, the scheme in this paper is quite similar to the variational multi-scale schemes in the two references. The difference is the schemes presented in \cite{LARSON, LARSON1} are adaptive schemes based upon some a posterior error estimates and therefore some boundaries related problems have to be solved. Another contribution of this paper compared with \cite{LARSON, LARSON1} is that a priori error estimate of the scheme is obtained and for patches of given size, our analysis show that a few iterations, say $O(|\ln H|^2)$ or $O(|\ln H|)$ in 2-D or 3-D respectively, will generate an approximation with same accuracy as the fine mesh standard Galerkin approximation. In addition, following the idea of partition of unity method (see \cite{BABUSKA}), authors in \cite{WANG, ZHENG} proposed a local and parallel two-grid scheme for second order linear elliptic equations in 2-D case based on the scheme presented in \cite{XU1}. Although the usage of partition of unity method makes the global approximation continuous, but their error estimation is still based on the superapproximation property of the finite element space and therefore the distance $t$, theoretically, must be constant order to guarantee their estimations.  

The rest of this paper is organized as follows. In the coming section, some preliminary materials are provided. In section 3, local and parallel scheme
is constructed. Error estimates in both $H^1$ and $L^2$ norms are obtained for the scheme in section 4. Finally, some numerical experiments are given to support our analysis in section 5.

\section{Preliminaries}

In this paper, for the sake of simplicity of analysis, we only consider the case of the Poisson equation (\ref{equ}) and we can get similar results for general linear elliptic problems as the Poisson equation with little modifications. 

For a bounded convex domain $\Omega\subset R^d$, $d=2,3$, we use the standard notations for Sobolev spaces $W^{s,p}(\Omega)$ and their
associated norms, see, e.g., \cite{ADAMS} and \cite{CIARLET}. For $p=2$, we denote $H^s(\Omega)=W^{s,2}(\Omega)$ and
$H_0^1(\Omega)=\{v\in H^1(\Omega): v|_{\partial\Omega}=0\}$, $\|\cdot\|_{s,\Omega}=\|\cdot\|_{s,2,\Omega}$ and $|\cdot|_{s,\Omega}$ the corresponding semi-norm.
In some places of this paper, $\|\cdot\|_{s,\Omega}$ should be viewed as piecewisely defined if it is necessary. For simplicity, the following symbols $\lesssim$, $\gtrsim$ and $\approxeq$ will be used in this paper. In the rest, $x_1\lesssim y_1$, $x_2\gtrsim y_2$ and $x_3\approxeq y_3$,
mean that $x_1\leq C_1y_1$, $x_2\geq c_2y_2$ and $c_3x_3\leq y_3\leq C_3x_3$ for some constants $C_1$, $c_2$, $c_3$ and $C_3$ that are independent of mesh
size, $x_i$, $y_i$ and local domains which will be introduced in the following sections. In the following, we denote by $(\cdot,\cdot)$ the $L^2-$inner product on $\Omega$. Thus, $\|\cdot\|_{0,\Omega}=(\cdot,\cdot)^\frac12$
and, in $H^1_0(\Omega)$, we know that $\|\cdot\|_{1,\Omega}\approxeq \|\nabla\cdot\|_{0,\Omega}$. For simplicity of expression, we use $\|\cdot\|_{\Omega}$ to denote $\|\cdot\|_{0,\Omega}$ in the rest. For subdomains $S_1\subset S_2\subset\Omega$, $S_1\subset\subset S_2$ means that $\mbox{dist}(\partial S_2\backslash\partial\Omega, \partial S_1\backslash\partial\Omega)>0$. 

For any given $S_1\subset\Omega$, we denote by $(\cdot,\cdot)_{S_1}$ the $L^2-$inner product on $S_1$
$$
a(u,v)_{S_1}=(\nabla u,\nabla v)_{S_1},\quad a(u,v)=(\nabla u,\nabla v)_\Omega,
$$
in the rest of this paper. Then we get the weak form (\ref{wequ}).

It is obvious that
\begin{equation}\label{CC}
\|\nabla u\|_{S_1}^2=a(u,u)_{S_1}\leq \|\nabla u\|_{S_1}\|\nabla v\|_{S_1}.
\end{equation}

We assume that $T^H(\Omega)=\{\tau_\Omega^H\}$ is a regular triangulation of $\Omega$. Here
$H=\max\limits_{\tau_\Omega^H\in T^H(\Omega)}\{\mbox{diam}(\tau_\Omega^H)\}$ is the mesh size parameter.
Let
$$S^H(\Omega)=\{v_H\in C^0(\Omega): v_H|_{\tau_\Omega^H}\in P^r_{\tau_\Omega^H},\;\forall\tau_\Omega^H\in T^H(\Omega)\},$$
be a $C^0-$finite element space defined on $\Omega$ and $S_0^H(\Omega)=S^H(\Omega)\cap H^1_0(\Omega)$,
where $r\geq 1$ is a positive integer and $P^r_{\tau_\Omega^H}$ is the space
of polynomials of degree not greater than $r$ defined on $\tau_\Omega^H$.
Given $S_1\subset\Omega$, which aligns with $T^H(\Omega)$, we define $T^H(S_1)$ and $S^H(S_1)$ to be the restriction
of $T^H(\Omega)$ and $S^H(\Omega)$ on $S_1$.

For these finite element spaces and problem (\ref{wequ}), we make the following assumptions.
\begin{itemize}
\item[{\bf A1}] \emph{Interpolant}. There is a finite element interpolation $I_H$ defined on $S^H(\Omega)$ and we denote
 $\hat I_H=I-I_H$ such that for any $w\in H^s(\tau_\Omega^H)$, $0\leq m\leq s\leq r+1$,
$$
\|\hat I_Hw\|_{m,\tau_\Omega^H}\lesssim H^{s-m}|w|_{s,\tau_\Omega^H}.
$$
\item[{\bf A2}.] \emph{Inverse Inequality}. For any $w\in S^H(\Omega)$,
$$\|w\|_{1,\Omega}\lesssim H^{-1}\|w\|_{\Omega},$$
\item[{\bf A3}.] \emph{Regularity}. For any $f\in L^2(\Omega)$,  the solutions of
$$
a(u,v)=(f,v)\quad\forall v\in H_0^1(\Omega),
$$
satisfies
$$
\|u\|_{2,\Omega}\lesssim \|f\|_{L^2(\Omega)}.
$$
\end{itemize}


Now let us state the standard Galerkin equation of (\ref{wequ}): find $u_H\in S_0^H(\Omega)$ such that
\begin{equation}\label{SGM}
a(u_H,v_H)=(f,v_H),\quad\forall v_H\in S_0^H(\Omega).
\end{equation}
And it is classical that
\begin{equation}\label{ASGM}
\|u-u_H\|_{\Omega}+H\|\nabla(u-u_H)\|_{\Omega}=O(H^{r+1}),
\end{equation}
if $u\in H_0^1(\Omega)\cap H^{r+1}(\Omega)$, $r\geq 1$.

\section{Local and Parallel Two-Grid Scheme}

Let us denote
$$
\hat w=u-u_H\in H_0^1(\Omega).
$$
Then the residual equation is
\begin{equation}\label{errequ}
a(\hat w,v)=(f,v)-a(u_H,v),\quad\forall v\in H_0^1(\Omega).
\end{equation}
Assume that $\{\phi_j\}_{j=1}^N$ is a partition of unity on $\Omega$ for given integer $N\geq 1$ such that
$\Omega\subset\bigcup\limits_{j=1}^N\mbox{supp }\phi_j$
and $\sum\limits_{j=1}^N\phi_j\equiv 1$ on $\Omega$. In the rest of the paper, we denote $D_j=\mbox{supp}\,\phi_j$ and always assume that $D_j$ aligns with $T^H(\Omega)$. We can rewrite (\ref{errequ}) as
\begin{equation*}
a(\hat w,v)=(f,\sum\limits_{j=1}^N\phi_j v)-a(u_H,\sum\limits_{j=1}^N\phi_j v),\quad\forall v\in H_0^1(\Omega).
\end{equation*}
By superposition principle, the above residual equation is equivalent to the summation of the following subproblems:
\begin{equation}\label{errequ2}
a(\hat w^j,v)=(f,\phi_j v)-a(u_H,\phi_j v),\quad\forall v\in H_0^1(\Omega),\;j=1,2,\cdots,N.
\end{equation}
That is $\hat w=\sum\limits_{j=1}^N \hat w^j$. Each subproblem is a "local residual" equation with homogeneous Dirichlet boundary condition, which is driven by right-hand-side term of a very
small compact support, and all the subproblems are independent once $u_H$ is known. To discretize and localize the "local
residual" equation and therefore to reduce the computational scale, we restrict the above subproblem in a local domain $\Omega_j$, which
contains $D_j$ and is also assumed to be aligned with $T^H(\Omega)$.
For each local subdomain $\Omega_j$, we assume that $T^h(\Omega_j)=\{\tau^h_{\Omega_j}\}$ is a regular triangulation
on it. Here $h=\max\limits_{1\leq j\leq N}\max\limits_{\tau^h_{\Omega_j}\in T^h(\Omega_j)}\{\mbox{diam}(\tau^h_{\Omega_j})\}$. For simplicity, the local fine mesh $T^h(\Omega_j)$ is defined as follows throughout the rest of the paper. For a global regular triangulation $T^h(\Omega)=\{\tau^h_\Omega\}$  on $\Omega$ which aligns with $T^H(\Omega)$, we define $T^h(\Omega_j)=T^h(\Omega)|_{\Omega_j}$. For this mesh
parameter $h$($h<H$), we introduce following fine mesh
finite element spaces $S^{h}(\Omega_j)$, $S_0^{h}(\Omega_j)$ and $S^h(\Omega)$, $S_0^h(\Omega)$, which have the
same definitions as $S^H(\Omega)$ and $S_0^H(\Omega)$ given in the previous section. Since the functions in
$S_0^h(\Omega_j)$ can be extended to functions in $S_0^h(\Omega)$ with zero value outside $\Omega_j$, we regard
$S_0^h(\Omega_j)$ as a subspace of $S_0^h(\Omega)$ in the sense of such zero extension.
Since $\Omega_j$ and $T^h(\Omega)$ align with $T^H(\Omega)$, we always assume
\begin{equation}\label{hierarchical}
S^H(\Omega)\subset S^h(\Omega),\quad S_0^H(\Omega)\subset S_0^h(\Omega)=\bigcup\limits_{1\leq j\leq N}S_0^h(\Omega_j).
\end{equation}

Now we give the approximate "local residual" equation as follows: find
$\hat w^j_{H,h}\in S_0^{h}(\Omega_j)$ such that
\begin{equation}\label{alerrequ}
a(\hat w^j_{H,h},v)=(f,\phi_j v)-a(u_H,\phi_j v),\quad\forall v\in S_0^h(\Omega_j),\;j=1,2,\cdots,N.
\end{equation}
It is clear that all subproblems in (\ref{alerrequ}) are independent. Note that $\hat w^j_{H,h}$ can be extended to the entire domain $\Omega$ with zero value outside $\Omega_j$ in $H_0^1(\Omega)$, we still use
$\hat w^j_{H,h}$ to denote such extension in the rest and we denote
$$\hat w_{H,h}=\sum\limits_{j=1}^N \hat w^j_{H,h}.$$
Now we define the following intermediate approximate solution
\begin{equation}\label{post1}
u_{H,h}=u_H+\hat w_{H,h}.
\end{equation}

Since the approximation $u_{H,h}$ is obtained by solving a series of local subproblems which are imposed
with artificial homogeneous boundary conditions of the first kind, some local non-physical oscillation may occur. This
will certainly bring some bad influence to the global accuracy of the approximation. To diminish such influence,
we choose to smooth the above intermediate approximation $u_{H,h}$ by following coarse grid correction:
find $E_H\in S_0^H(\Omega)$ such that
\begin{equation}\label{post2}
a(E_H,v)=(f,v)-a(u_{H,h},v)\quad\forall v\in S_0^H(\Omega).
\end{equation}
And the final approximate solution is defined as
\begin{equation}\label{final}
u_H^h=u_{H,h}+E_H=u_H+\hat w_{H,h}+E_H.
\end{equation}

Now for the implementation of the proposed local and parallel two-grid
scheme (\ref{alerrequ})-(\ref{final}), we have to choose a proper partition of unity $\{\phi_j\}_{j=1}^N$ of $\Omega$ and its associated computational domain $\Omega_j$. A simple choice of the partition of unity is the piecewise linear Lagrange basis functions associated with the coarse grid triangulation $T^H(\Omega)$, where $N$ is the number of vertices of $T^H(\Omega)$ including the boundary vertices. 
For each vertex $j$ of the coarse grid, let us denote $D_j=\mbox{supp }\phi_j$. Then we expand $D_j$ by one coarse mesh layer to get $\Omega_j$, that is
$$
\Omega_j=\bigcup_{x_i\in D_j} D_i,
$$
where $x_i$ denotes the $i$th vertex of the coarse mesh triangulation. It is clear that
\begin{equation}\label{3281411}
\mbox{diam}(D_j),\;\mbox{dist}(\partial D_j,\partial\Omega_j)\approxeq H.
\end{equation}

\section{Error Estimates}
By the idea of fictitious domain method (see \cite{GLOWINSKI}), we extend the local sub-problem (\ref{alerrequ}) to $\Omega$.
Let us denote $\Gamma=\partial\Omega$ and $\Gamma_j=\partial\Omega_j\backslash\Gamma$.  If we introduce $H_h^\frac12(\Gamma_j)=S_0^h(\Omega)|_{\Gamma_j}\subset H^\frac12(\Gamma_j)$ and $H^{-\frac12}_h(\Gamma_j)=(H^\frac12_h(\Gamma_j))'$ 
which is equipped with the following norm
$$
\|\mu\|_{H^{-\frac12}_h(\Gamma_j)}=\sup\limits_{v\in H_h^\frac12(\Gamma_j)}\frac{\int_{\Gamma_j}v\mu}{\|v\|_{H_h^\frac12(\Gamma_j)}},
$$
we can show that the local residual $\hat w_{H,h}^j\subset S_0^h(\Omega)$ satisfies the following saddle point problem: find $(\hat w_{H,h}^j,\xi^j)\in S_0^h(\Omega)\times H_h^{-\frac12}(\Gamma_j)$ such that $\forall (v,\mu)\in S_0^h(\Omega)\times H_H^{-\frac12}(\Gamma_j)$
\begin{equation}\label{fictitious}
a(\hat w^j_{H,h},v)+<\xi^j,v>_j+<\mu,\hat w^j_{H,h}>_j=(f,\phi_j v)-a(u_H,\phi_jv),
\end{equation}
where
$$
<\mu,v>_j=\int_{\Gamma_j}\mu vds\quad\forall \mu\in H^{-\frac12}_h(\Gamma_j),\;v\in S_0^h(\Omega).
$$

To show the well-posedness of the above saddle point problem,
let us introduce following two finite element spaces
\begin{eqnarray*}
&&S_E^h(\Omega_j)=\{v\in S^h(\Omega_j): v|_{\partial\Omega_j\backslash\Gamma_j}=0\},\\
&&S_E^h(\Omega\backslash\Omega_j)=\{v\in S^h(\Omega\backslash\Omega_j): v|_{\partial(\Omega\backslash\Omega_j)\backslash\Gamma_j}=0\}.
\end{eqnarray*}
For given $g\in H_h^\frac12(\Gamma_j)$, we introduce two auxiliary problems
$$
a(u_1,v)_{\Omega_j}=0,\quad u_1|_{\Gamma_j}=g\quad\forall v\in S_0^h(\Omega_j),
$$
and
$$
a(u_2,v)_{\Omega\backslash\Omega_j}=0,\quad u_2|_{\Gamma_j}=g\quad\forall v\in S_0^h(\Omega\backslash\Omega_j).
$$
These two problems define two mappings $\gamma_1^{-1}$ and $\gamma_2^{-1}$ from $H_h^\frac12(\Gamma_j)$ into $S_E^h(\Omega_j)$ and $S_E^h(\Omega\backslash\Omega_j)$, respectively. That is
$$
u_1=\gamma_1^{-1}g,\quad u_2=\gamma_2^{-1}g.
$$
And we know that
$$
\|\gamma_1^{-1}g\|_{H^1(\Omega_j)},\;\|\gamma_2^{-1}g\|_{H^1(\Omega\backslash\Omega_j)}\lesssim \|g\|_{H^\frac12(\Gamma_j)}.
$$
Then we can define an operator $\gamma^{-1}$ from $H_h^\frac12(\Gamma_j)$ into $S_0^h(\Omega)$: for any given $g\in H_h^\frac12(\Gamma_j)$
$$
\gamma^{-1}g=\left\{\begin{array}{ll}
\gamma_1^{-1}g,\quad &\mbox{in}\;\Omega_j,\\
\gamma_2^{-1}g,\quad &\mbox{in}\;\Omega\backslash\Omega_j.
\end{array}\right.
$$
And we have the following property of $\gamma^{-1}$:
$$
\|\gamma^{-1}g\|_{H_0^1(\Omega)}\lesssim \|g\|_{H^\frac12(\Gamma_j)}\quad\forall g\in H_h^\frac12(\Gamma_j).
$$

Now for any $\mu\in H_h^{-\frac12}(\Gamma_j)$, we have
\begin{eqnarray*}
&&\|\mu\|_{H_h^{-\frac12}(\Gamma_j)}=\sup\limits_{g\in H_h^\frac12(\Gamma_j)}\frac{<\mu,g>_j}{\|g\|_{H^\frac12(\Gamma_j)}}\lesssim
\sup\limits_{g\in H_h^\frac12(\Gamma_j)}\frac{<\mu,\gamma^{-1}g>_j}{\|\gamma^{-1}g\|_{H_0^1(\Omega)}}
\leq\sup\limits_{v\in S_0^h(\Omega)}\frac{<\mu,v>_j}{\|v\|_{H_0^1(\Omega)}}.
\end{eqnarray*}
This ensures that the saddle point problem is well-posed. And it is straight that $\hat w_{H,h}^j$ is the solution of this global problem.

Let us recall the residual equation (\ref{errequ}) and the "local residual" equations (\ref{errequ2}). For the previously defined fine mesh $T^h(\Omega)$ and the associated finite element space $S_0^h(\Omega)$, their fine mesh Galerkin approximations are as follows.
Find $\hat w_H\in S_0^h(\Omega)$ and $\hat w_H^j\in S_0^h(\Omega)$, $j=1,2,\cdots,N$, such that
\begin{equation}\label{errequh}
a(\hat w_H,v)=(f,v)-a(u_H,v),\quad\forall v\in S_0^h(\Omega),
\end{equation}
and
\begin{equation}\label{errequh2}
a(\hat w^j_H,v)=(f,\phi_j v)-a(u_H,\phi_j v),\quad\forall v\in S_0^h(\Omega),\;j=1,2,\cdots,N.
\end{equation}
And we know that $\hat w_H=\sum\limits_{j=1}^N\hat w_H^j$ is the Galerkin approximation of $\hat w$ in $S_0^h(\Omega)$.

If we denote
$$
g_j=\hat w_H^j|_{\Gamma_j},
$$
we know that $\hat w_H^j$ satisfies: $\forall (v,\mu)\in S_0^h(\Omega)\times H_h^{-\frac12}(\Gamma_j)$
\begin{equation}\label{fictitious2}
a(\hat w^j_H,v)+<\zeta^j,v>_j+<\mu,\hat w^j_H-g_j>_j=(f,\phi_j v)-a(u_H,\phi_j v).
\end{equation}
Here, we can verify the Lagrange multiplier $\zeta^j$ satisfies
\begin{equation}\label{331850}
<\zeta^j,v>_j=0,\quad\forall v\in S_0^h(\Omega),
\end{equation}
if we take $\mu=0$ in (\ref{fictitious2}) and compare (\ref{fictitious2}) with (\ref{errequh2}).

If we denote by
$$
e^j_{H,h}=\hat w^j_H-\hat w^j_{H,h},\quad e_{H,h}=\sum\limits_{j=1}^N e^j_{H,h}=u_h-u_{H,h},
$$
the "local error" and the global error of $u_{H,h}$ respectively, then by comparing (\ref{fictitious2}) with (\ref{fictitious}) and taking $\mu=0$, we have
$$
a(e^j_{H,h},v)+<\xi^j,v>_j=0,\quad\forall v\in S_0^h(\Omega),\;j=1,2,\cdots,N,
$$
and
\begin{equation}\label{errorequ}
a(e_{H,h},v)+\sum\limits_{j=1}^N<\xi^j,v>_j=0,\quad\forall v\in S_0^h(\Omega).
\end{equation}

It is clear that for all points $x\in \Omega$, there exists a positive integer $\kappa$, which has nothing to do with $N$ and $x$, such that each $x$ belongs to $\kappa$  different $\Omega_j$ at most. 
By using the previously defined operator $\gamma^{-1}$ and the fact we just stated, we can easily get the following lemma.
\begin{lemma}\label{LE0}{\em The multiplier $\xi^j$ in (\ref{fictitious}) satisfies
$$
\|\xi^j\|_{H_h^{-\frac12}(\Gamma_j)}\lesssim\|\nabla e_{H,h}^j\|_\Omega,
$$
and
$$
\sum\limits_{j=1}^N<\xi^j,v>_j\lesssim \kappa^\frac12(\sum\limits_{j=1}^N\|\xi\|^2_{H_h^{-\frac12}(\Gamma_j)})^\frac12\|v\|_{H^1(\Omega)}.
$$
}
\end{lemma}
\begin{proof} The first estimate is quite easy if one notices the property of $\gamma^{-1}$ and we omit its proof. For the second estimate, 
thanks to the definition of $\|\cdot\|_{H_h^{-\frac12}(\Gamma_j)}$,
we have $\forall v\in S_0^h(\Omega)$
\begin{eqnarray*}
&&\sum\limits_{j=1}^N<\xi^j,v>_j\leq\sum\limits_{j=1}^N\|\xi^j\|_{H_h^{-\frac12}(\Gamma_j)}\|v\|_{H^\frac12(\Gamma_j)}\\
&&\qquad\lesssim \sum\limits_{j=1}^N\|\xi\|_{H_h^{-\frac12}(\Gamma_j)}\|v\|_{H^1(\Omega_j)}\leq(\sum\limits_{j=1}^N\|\xi\|^2_{H_h^{-\frac12}(\Gamma_j)})^\frac12(\sum\limits_{j=1}^N\|v\|^2_{H^1(\Omega_j)})^\frac12\\
&&\qquad\leq\kappa^\frac12(\sum\limits_{j=1}^N\|\xi\|^2_{H_h^{-\frac12}(\Gamma_j)})^\frac12\|v\|_{H^1(\Omega)}.
\end{eqnarray*}
\end{proof}

Now let us consider the estimation of each $\|\nabla e_{H,h}^j\|_{\Omega}$. To do so, we notice that
$$
\|\nabla e_{H,h}^j\|_{\Omega}^2=\|\nabla(\hat w_H^j-\hat w_{H,h}^j)\|_{\Omega}^2=\|\nabla\hat w_H^j\|_{\Omega/\Omega_j}^2+\|\nabla(\hat w_H^j-\hat w_{H,h}^j)\|_{\Omega_j}^2.
$$
It is obvious that $e_{H,h}^j|_{\Omega_j}$ satisfies the following equation
$$
a(e_{H,h}^j|_{\Omega_j},v)=0,\quad v\in S_0^h(\Omega_j),\quad e_{H,h}^j|_{\partial\Omega_j}=\hat w_H^j|_{\partial\Omega_j}.
$$
Since $\hat w_H^j\in S_0^h(\Omega)$, we have
$$
\|e_{H,h}^j\|_{1,\Omega_j}\lesssim \|\hat w_H^j\|_{\frac12,\partial\Omega_j}=\|\hat w_H^j\|_{\frac12,\partial(\Omega\backslash\Omega_j)}\lesssim \|\nabla\hat w_H^j\|_{\Omega/\Omega_j}.
$$
Then we know that
\begin{equation}\label{S1}
\|\nabla e_{H,h}^j\|_{\Omega}^2\lesssim\|\nabla\hat w_H^j\|_{\Omega/\Omega_j}^2.
\end{equation}

Being aware of (\ref{S1}), we give the estimate of $\|\nabla e_{H,h}^j\|_{\Omega}$, which plays a crucial role in this section.

\begin{lemma}\label{LE2}
{\em Let us denote
$$
\alpha_d=\left\{\begin{array}{ll}
\frac{c}{|\ln H|^2}, \quad & d=2,\\
\frac{c}{|\ln H|}, & d=3,
\end{array}\right.
$$
where $c>0$ is a positive constant that does not depend on $H$, $h$ and $\Omega_j$. Then we have
$$
\|\nabla e_{H,h}^j\|_{\Omega}^2\lesssim H^{2\alpha_d}\|\nabla\hat w_H^j\|_\Omega^2.
$$
}\end{lemma}
\begin{proof} 
To prove this lemma, we first divide the region $\Omega_j\backslash D_j$ as follows. For example, see Fig. \ref{MYFIG}. 
\begin{center}
 \begin{figure}[ht]
  \centering
   \includegraphics[width=0.8\linewidth]{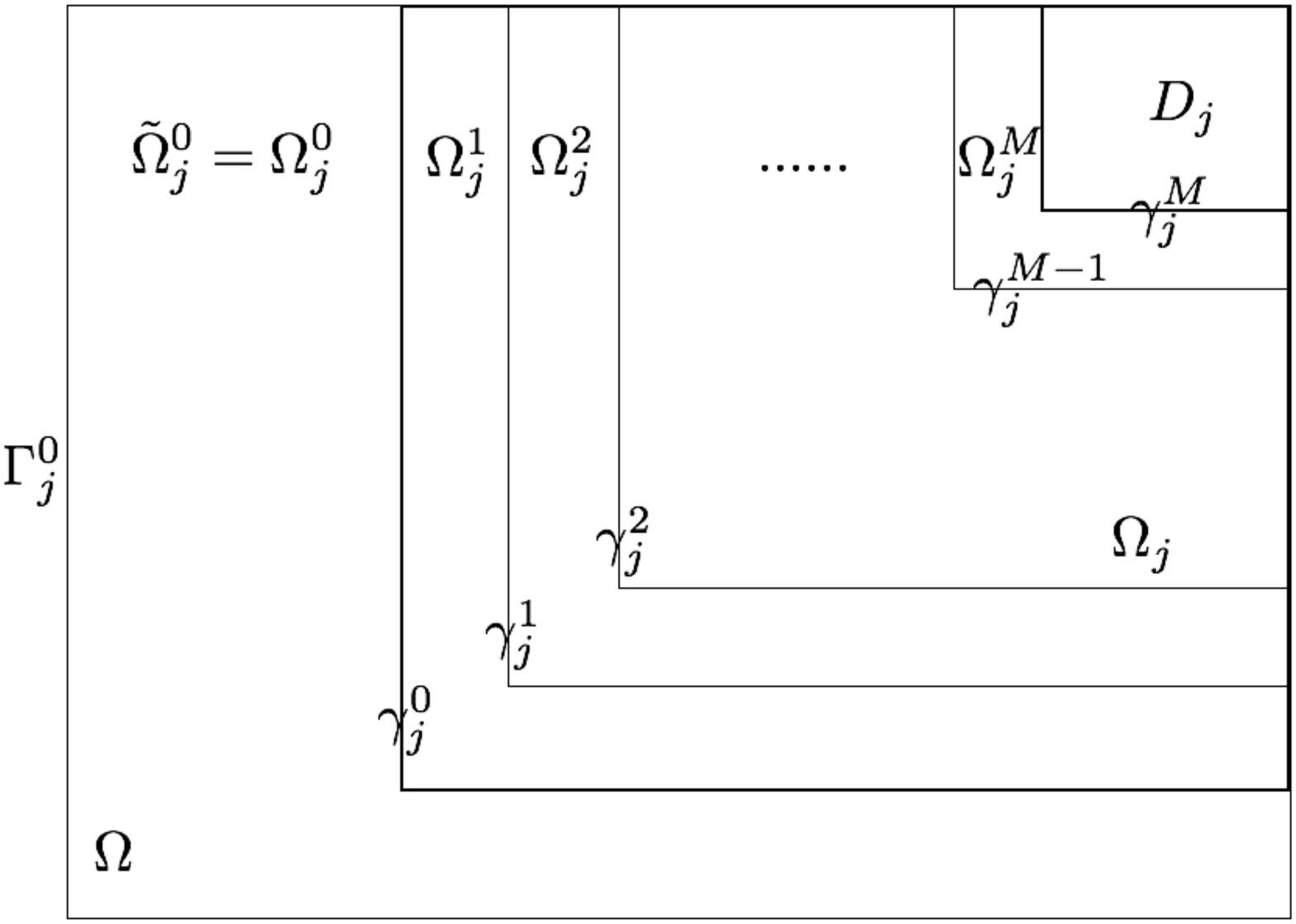}
   \caption{Division of the region $\Omega_j\backslash D_j$}\label{MYFIG}
 \end{figure}
\end{center}
Let us denote
$$
\tilde\Omega_j^0=\Omega_j^0=\Omega\backslash\Omega_j.
$$
We extend the domain $\tilde\Omega_j^0$ along the outward normal direction on $\partial\tilde\Omega_j^0\backslash\partial\Omega$ within $\Omega$ by a fine mesh layer to obtain $\tilde\Omega_j^1\supset\tilde\Omega_j^0$ and denote by $\Omega_j^1$ the incremental annular zone, that is $\Omega_j^1=\tilde\Omega_j^1\backslash\tilde\Omega_j^0$. Repeat the above procedure until we get $\tilde\Omega_j^M=\Omega\backslash D_j$, where $M\approxeq \frac{H}{h}$. 
Then we obtain a series of subdomains
$$
\tilde\Omega_j^0\subset\subset\tilde\Omega_j^1\subset\subset\cdots\subset\subset\tilde\Omega_j^M,
$$
and a series of disjoint annular zones
$$
\Omega_j^0,\Omega_j^1,\cdots,\Omega_j^M.
$$
It is clear that
$$
\tilde\Omega_j^k=\bigcup\limits_{i=0}^k\Omega_j^i,\quad k=0,1,2,\cdots,M.
$$
In what follows, we denote
$$
\partial\tilde\Omega_j^k=\gamma_j^k\cup\Gamma_j^k,\quad\gamma_j^k=\partial\tilde\Omega_j^k/\partial\Omega\quad\mbox{and}\quad\Gamma_j^k=\partial\tilde\Omega_j^k\backslash\gamma_j^k,\quad k=0,1,2,\cdots,M.
$$

Since
$$
a(\hat w_H^j,v)_{\tilde\Omega_j^k}=0,\quad\forall v\in S^h_0(\tilde\Omega_j^k),\quad k=1,2,\cdots,M,
$$
and $\tilde\Omega_j^k=\tilde\Omega_j^{k-1}\cup\Omega_j^k$, we know that
\begin{equation}\label{328942}
a(\hat w_H^j,v)_{\tilde\Omega_j^{k-1}}=-a(\hat w_H^j,v)_{\Omega_j^k},\quad v\in S_0^h(\tilde\Omega_j^k).
\end{equation}


We define a smooth function $\psi\in S^h(\tilde\Omega_j^k)$ such that $\psi|_{\gamma_j^k}=0$ for $k\geq 1$ and
$$
\mbox{supp}\,\psi=\tilde\Omega_j^k,\quad \psi(x)\equiv 1\;\forall x\in\tilde\Omega_j^{k-1},\quad
0\leq\psi\leq 1\quad\mbox{and}\quad|\nabla\psi(x)|\lesssim h^{-1}.
$$
By taking
$$
v=I_h(\psi\hat w_H^j)\in S_0^h(\tilde\Omega_j^k),
$$
in (\ref{328942}), we derive
\begin{eqnarray}\label{3281446}
\hskip-1cm\|\nabla\hat w_H^j\|^2_{\tilde\Omega_j^{k-1}}&&=a(\hat w_H^j,\hat w_H^j)_{\tilde\Omega_j^{k-1}}=-a(\hat w_H^j,I_h(\psi\hat w_H^j))_{\Omega_j^k}\\
\nonumber &&
\leq \|\nabla\hat w_H^j\|_{\Omega_j^k}\|\nabla I_h(\psi\hat w_H^j)\|_{\Omega_j^k}\lesssim \|\nabla\hat w_H^j\|_{\Omega_j^k}\|\nabla (\psi\hat w_H^j)\|_{\Omega_j^k}.
\end{eqnarray}

For $\|\nabla(\psi\hat w_H^j)\|_{\Omega_j^k}$, we have
\begin{equation}\label{02031123}
\|\nabla(\psi\hat w_H^j)\|_{\Omega_j^k}\leq\|\psi\nabla\hat w_H^j\|_{\Omega_j^k}+\|\hat w_H^j\nabla\psi\|_{\Omega_j^k}
\leq\|\nabla\hat w_H^j\|_{\Omega_j^k}+h^{-1}\|\hat w_H^j\|_{\Omega_j^k}.
\end{equation}

Now let us estimate $\|\hat w_H^j\|_{\Omega_j^k}$. To do so, we introduce the following polar or spherical coordinates $(\rho,\omega)$ with origin at $j$th vertex with respect to 2-D and 3-D case, respectively. Here $\omega=\omega_1$ in 2-D case and $\omega=(\omega_1,\omega_2)$ in 3-D case. The Jacobi determinant of the transformation between the Cartesian coordinate $(x_1,\cdots, x_d)$ and $(\rho,\omega)$ is
$$
J=\frac{D(x_1,\cdots,x_d)}{D(\rho,\omega_1\cdots,\omega_{d-1})}=\rho^{d-1}\delta(\omega),\quad \delta(\omega)=\sin^{d-2}\omega_1.
$$

In the polar or spherical coordinates, let us denote
$$
\Gamma_j^k:\rho(\omega)=\rho_j^k(\omega),\quad \gamma_j^k:\rho(\omega)=\tilde\rho_j^k(\omega),\quad 0\leq k\leq M.
$$
It is obvious that 
$$
\tilde\rho_j^k(\omega)\approxeq H\quad\mbox{and}\quad H\lesssim\rho_j^k(\omega)\lesssim 1.
$$

For $1 \leq k\leq M$ and any point $(\rho,\omega)\in\Omega_j^k$, noting $\hat w_H^j\in H_0^1(\Omega)$ and $\Omega$ is a convex domain, we have
\begin{eqnarray*}
&&|\hat w_H^j(\rho,\omega)|=|\int_{\rho_j^k(\omega)}^{\rho}\frac{\partial\hat w_H^j}{\partial\rho}d\rho|
\leq (\int_{\rho_j^k}^{\tilde\rho_j^k}\frac{1}{\rho^{d-1}}d\rho)^\frac12(\int^{\tilde\rho_j^k}_{\rho_j^k}\rho^{d-1}|\frac{\partial\hat w_H^j}{\partial\rho}|^2d\rho)^\frac12\\
&&\qquad \lesssim \beta_d^\frac12(H)(\int^{\tilde\rho_j^k}_{\rho_j^k}\rho^{d-1}|\frac{\partial\hat w_H^j}{\partial\rho}|^2d\rho)^\frac12,
\end{eqnarray*}
where $\beta_d(H)=H^{-1}$ when $d=3$ and $\beta_d(H)=|\ln H|$ when $d=2$. Thus
\begin{eqnarray*}
&&\|\hat w_H^j\|_{\Omega_j^k}^2=\int_{\Omega_j^k}\rho^{d-1}|\hat w_H^j(\rho,\omega)|^2\delta(\omega)d\rho d\omega\\
&&\qquad\lesssim \int_{\Omega_j^k}\rho^{d-1}\beta_d(H)(\int_{\rho_j^k}^{\tilde\rho_j^k}\rho^{d-1}|\frac{\partial\hat w_H^j}{\partial\rho}|^2 d\rho)
\delta(\omega)d\rho d\omega\\
&&\qquad\lesssim hH^{d-1}\beta_d(H)\|\nabla\hat w_H^j\|^2_{\tilde\Omega_j^k}.
\end{eqnarray*}
Since $\|\nabla\hat w_H^j\|_{\tilde\Omega_j^k}\leq\|\nabla\hat w_H^j\|_{\Omega_j^k}+\|\nabla\hat w_H^j\|_{\tilde\Omega_j^{k-1}}$ and $h^{-\frac12}H^\frac{d-1}{2}\beta_d^\frac12(H)>1$, combining the above estimate with (\ref{3281446}) and (\ref{02031123}) admits
\begin{eqnarray*}
&&\|\hat w_H^j\|_{\tilde\Omega_j^{k-1}}^2\lesssim \|\nabla\hat w_H^j\|_{\Omega_j^k}(\|\nabla\hat w_H^j\|_{\Omega_j^k}+h^{-\frac12}H^\frac{d-1}{2}\beta_d^\frac12(H)\|\nabla\hat w_H^j\|_{\tilde\Omega_j^k})\\
&&\qquad\leq \|\nabla\hat w_H^j\|_{\Omega_j^k}^2+h^{-\frac12}H^\frac{d-1}{2}\beta_d^\frac12(H)\|\nabla\hat w_H^j\|_{\Omega_j^k}(\|\nabla\hat w_H^j\|_{\Omega_j^k}+\|\nabla\hat w_H^j\|_{\tilde\Omega_j^{k-1}})\\
&&\qquad\lesssim h^{-\frac12}H^\frac{d-1}{2}\beta_d^\frac12(H)\|\nabla\hat w_H^j\|_{\Omega_j^k}^2+h^{-\frac12}H^\frac{d-1}{2}\beta_d^\frac12(H)\|\nabla\hat w_H^j\|_{\Omega_j^k}\|\nabla\hat w_H^j\|_{\tilde\Omega_j^{k-1}}.
\end{eqnarray*}
By Young's inequality, we have
$$
\|\hat w_H^j\|_{\tilde\Omega_j^{k-1}}^2\lesssim h^{-1}H^{d-1}\beta_d(H)\|\nabla\hat w_H^j\|_{\Omega_j^k}^2,
$$
or
$$
\|\nabla\hat w_H^j\|_{\tilde\Omega_j^k}^2\geq chH^{1-d}\beta_d^{-1}(H)\|\nabla\hat w_H^j\|_{\tilde\Omega_j^{k-1}}^2,
$$
where $c>0$ is a constant that does not depend on $H$, $h$, $j$ and $k$. 

By using the last inequality successively, we get
\begin{eqnarray*}
&&\|\nabla\hat w_H^j\|_{\tilde\Omega_j^M}^2=\|\nabla\hat w_H^j\|_{\tilde\Omega_j^{M-1}}^2+
\|\nabla\hat w_H^j\|_{\Omega_j^M}^2\geq (1+chH^{1-d}\beta_d^{-1}(H))\|\nabla\hat w_H^j\|_{\tilde\Omega_j^{M-1}}^2\\
&& \qquad\geq\cdots\geq (1+chH^{1-d}\beta_d^{-1}(H))^M\|\nabla\hat w_H^j\|_{\tilde\Omega_j^0}^2.
\end{eqnarray*}
Hence
$$
\|\nabla\hat w_H^j\|_{\tilde\Omega_j^0}^2\leq (1+chH^{1-d}\beta_d^{-1}(H))^{-M}\|\nabla\hat w_H^j\|_{\tilde\Omega_j^M}^2.
$$

Noting $M\approxeq\frac{H}{h}$ and $\|\nabla\hat w_H^j\|_{\tilde\Omega_j^M}\leq\|\nabla\hat w_H^j\|_{\Omega}$, simple calculation shows that
$$
(1+chH^{1-d}\beta_d^{-1}(H))^{-M}\approxeq H^{2\alpha_d}.
$$
This, together with (\ref{S1}), concludes the proof of the lemma.
\end{proof}

Following the result in Lemma \ref{LE2}, to estimate $\|\nabla e_{H,h}^j\|_{\Omega}^2$, we have to give some estimates of the "local residual" $\hat w_H^j$.
\begin{lemma}\label{LE4} {\em Suppose the assumptions A1, A2 and A3 are valid. Then for $j=1,2,\cdots,N$, we have
$$\|\nabla\hat w_H^j\|_{\Omega}\lesssim \|\nabla(u-u_H)\|_{D_j}.$$
}
\end{lemma}

\begin{proof}
Thanks to the coercive property in (\ref{CC}), we can get from (\ref{errequ2}) that
\begin{eqnarray*}
\|\nabla\hat w_H^j\|_{\Omega}^2&&=a(\hat w_H^j,\hat w_H^j)=(f,\phi_j\hat w_H^j)-a(u_H,\phi_j\hat w_H^j)\\
&&=a(u-u_H,\phi_j\hat w_H^j)=a(u-u_H,\hat I_H(\phi_j\hat w_H^j))\\
&&=a(u-u_H,\hat I_H(\phi_jI_H\hat w_H^j)+\hat I_H[\phi_j\hat I_H\hat w_H^j]).
\end{eqnarray*}
From the continuity property of the bilinear form $a(\cdot,\cdot)$, we know
\begin{eqnarray*}
&&a(u-u_H,\hat I_H(\phi_jI_H\hat w_H^j)+\hat I_H[\phi_j\hat I_H\hat w_H^j])\\
&&\;\;\lesssim \|\nabla(u-u_H)\|_{D_j}(\|\nabla[\hat I_H(\phi_jI_H\hat w_H^j)]\|_{D_j}+\|\nabla \hat I_H[\phi_j\hat I_H\hat w_H^j]\|_{D_j}).
\end{eqnarray*}
Thanks to {\bf A1}, {\bf A2} and notice that $\phi_j$ is a linear function on each $\tau_\Omega^H$ and $|D\phi_j|\lesssim H^{-1}$,
\begin{eqnarray*}
&&\|\nabla[\hat I_H(\phi_jI_H \hat w_H^j)]\|_{D_j}=(\sum\limits_{\tau_\Omega^H\subset D_j}\|\nabla\hat I_H(\phi_jI_H \hat w_H^j)\|_{\tau_\Omega^H}^2)^\frac12\\
&&\qquad\leq H(\sum\limits_{\tau_\Omega^H\subset D_j}\|D^2(\phi_jI_H\hat w_H^j)\|_{\tau_\Omega^H}^2)^\frac12\\
&&\qquad\lesssim H(\sum\limits_{\tau_\Omega^H\subset D_j}[\|\phi_jD^2(I_H\hat w_H^j)\|_{\tau_\Omega^H}^2+\|D\phi_jD(I_H \hat w_H^j)\|_{\tau_\Omega^H}^2])^\frac12\\
&&\qquad\lesssim H(\sum\limits_{\tau_\Omega^H\subset D_j}[H^{-2}\|D(I_H\hat w_H^j)\|_{\tau_\Omega^H}^2+H^{-2}\|D(I_H\hat w_H^j)\|_{\tau_\Omega^H}^2])^\frac12\lesssim \|\nabla\hat w_H^j\|_{D_j},
\end{eqnarray*}
\begin{eqnarray*}
&&\|\nabla \hat I_H[\phi_j\hat I_H\hat w_H^j]\|_{D_j}\lesssim\|D(\phi_j\hat I_H\hat w_H^j)\|_{D_j}\\
&&\qquad\lesssim \|D\phi_j\hat I_H\hat w_H^j\|_{D_j}+\|\phi_jD(\hat I_H\hat w_H^j)\|_{D_j}\lesssim \|\nabla \hat w_H^j\|_{D_j}.
\end{eqnarray*}

Combination of the above estimates yields
$$
\|\nabla\hat w_H^j\|_{\Omega}\lesssim \|\nabla(u-u_H)\|_{D_j}.
$$

\end{proof}

Now we give the error estimations of the scheme (\ref{alerrequ})-(\ref{final}) in the following theorem.
\begin{theorem}\label{theorem1} {\em Suppose that assumptions A1, A2, A3 and (\ref{ASGM}) hold and $u\in H^{r+1}(\Omega)$. Then
$$
\|\nabla(u-u_H^h)\|_{\Omega}\lesssim h^r+H^{\alpha_d}\|\nabla(u-u_H)\|_{\Omega},
$$
$$
\|u-u_H^h\|_{\Omega}\lesssim h^{r+1}+H\|\nabla(u-u_H^h)\|_{\Omega},
$$
where $\alpha_d>0$ is defined in Lemma \ref{LE2}.
}
\end{theorem}

\proof
First of all, we introduce an $H^1-$orthogonal projection $P_H$ from $H_0^1(\Omega)$ onto $S_0^H(\Omega)$: for given $w\in H^1_0(\Omega)$, find $P_H w\in S_0^H(\Omega)$ such that
$$
a(v,w-P_H w)=0,\quad\forall v\in S_0^H(\Omega).
$$
It is classical that
$$
\|(I-P_H)w\|_{\Omega}\lesssim H\|\nabla w\|_{\Omega},\quad \forall w\in H^1_0(\Omega).
$$

Noticing the definition of $\hat w_{H,h}$ and $\sum\limits_{j=1}^N\phi_j=1$ in $\Omega$, the summation of all the equations of (\ref{fictitious}) with $\mu=0$ gives the equation satisfied by $\hat w_{H,h}$
\begin{equation}\label{1stcorrectionequation}
a(\hat w_{H,h},v) = (f,v)-a(u_H,v)-\sum\limits_{j=1}^N\int_{\Gamma_j}\xi^jvds,\quad\forall v\in S_0^h(\Omega).
\end{equation}
Then we have
$$
a(u_{H,h},v)=(f,v)-\sum\limits_{j=1}^N\int_{\Gamma_j}\xi^jvds,\quad\forall  v\in S_0^h(\Omega).
$$
Furthermore, we rewrite the coarse mesh correction as
$$
a(E_H,v)=(f,P_Hv)-a(u_{H,h},P_Hv),\quad\forall v\in S_0^h(\Omega).
$$

Adding the above two equations leads to 
$$
a(u_H^h,v)=(f,v)+(f,P_Hv)-a(u_{H,h},P_Hv)-\sum\limits_{j=1}^N\int_{\Gamma_j}\xi^jvds,\quad\forall v\in S_0^h(\Omega).
$$

Finally, we obtain
$$
a(u_H^h,v)=(f,v)-\sum\limits_{j=1}^N\int_{\Gamma_j}\xi^j(I-P_H)vds,\quad\forall v\in S_0^h(\Omega),
$$
and
\begin{equation}\label{errorequation}
a(u_h-u_H^h,v)=\sum\limits_{j=1}^N\int_{\Gamma_j}\xi^j(I-P_H)vds,\quad\forall v\in S_0^h(\Omega).
\end{equation}

Thanks to the above error equation of $u_h-u_H^h$, Lemma \ref{LE0}, \ref{LE2} and \ref{LE4}, we can easily get 
$$
\|\nabla(u_h-u_H^h)\|_\Omega\lesssim H^{\alpha_d}\|\nabla(u-u_H)\|_\Omega.
$$
Then we can derive the first result by using the triangle inequality.

For the $L^2-$error estimate, we use the Aubin-Nitsche duality argument. Since {\bf A3}, for $u_h-u_H^h\in L^2(\Omega)$, there exists $\phi\in H^2(\Omega)\cap H_0^1(\Omega)$ such that
$$
a(v,\phi)=(u_h-u_H^h,v),\quad\forall v\in H^1_0(\Omega),
$$
and
$$
\|\phi\|_{2,\Omega}\lesssim \|u_h-u_H^h\|_{\Omega}.
$$
Taking $v=u_h-u_H^h$ and noting (\ref{errorequation}), we have
$$
\|u_h-u_H^h\|_{\Omega}^2=a(u_h-u_H^h,\phi)=a(u_h-u_H^h,(I-P_H)\phi).
$$
Thus
\begin{eqnarray*}
\|u_h-u_H^h\|_{\Omega}^2=&&a(u_h-u_H^h,(I-P_H)\phi)\lesssim \|\nabla(u_h-u_H^h)\|_{\Omega}\|\nabla(I-P_H)\phi\|_{\Omega}\\
\lesssim  &&  H\|\nabla(u_h-u_H^h)\|_{\Omega}\|\phi\|_{2,\Omega}\lesssim H\|\nabla(u_h-u_H^h)\|_{\Omega}\|u_h-u_H^h\|_{\Omega}.
\end{eqnarray*}
By using triangle inequality, this estimate admits the $L^2-$error estimate.
\endproof

From the results in Theorem \ref{theorem1}, we see that one can improve the convergence order of both $H^1$ and $L^2$ errors of the coarse mesh standard Galerkin approximation $u_H$ for $\alpha_d$ order by one two-grid iteration.  And it is easy to verify that all the above lemmas and theorem are valid if we replace $u_H$ by $u_H^h$. This suggests the following two-grid iteration with
\begin{equation}\label{6210959}
K=[\alpha_d^{-1}+0.5]=\left\{\begin{array}{ll}
O(|\ln H|^2),\quad & d=2,\\
O(|\ln H|), & d=3.
\end{array}\right.
\end{equation}
\begin{description}
\item[(Step 0)] Let $k=0$ and solve (\ref{SGM}) to get $u_H\in S_0^H(\Omega)$ and we denote $u_H^{0,h}=u_H$;
\item[(Step 1)] Solve the equations in (\ref{alerrequ}) with $u_H=u_H^{k,h}$ to get $\{\hat w_{H,h}^{j}\}_{j=1}^N$, which are
denoted by $\{\hat w_{H,h}^{k+1,j}\}_{j=1}^N$ here. Then we get $u_{H,h}^{k+1}$ by (\ref{post1});
\item[(Step 2)] Solve (\ref{post2}) with $u_{H,h}=u_{H,h}^{k+1}$ to get $E_H^{k+1}$ and denote
$$
u_H^{k+1,h}=u_{H,h}^{k+1}+E_H^{k+1};
$$
If $k+1> K$, stop the iteration and denote $u_H^h=u_H^{k+1,h}$, which is the final approximation with optimal error. Otherwise, let $k:=k+1$ and goto (Step 1).
\end{description}


\begin{corollary}\label{corollary1} {\em Suppose $u\in H_0^1(\Omega)\cap H^{r+1}(\Omega)$, the final approximation $u_H^h$ of the scheme (Step 0)$\sim$(Step 2) has the following error bounds
\begin{eqnarray}
\label{ErrorH1} && \|\nabla(u-u_H^h)\|_{\Omega}\lesssim h^r+H^{r+1},\\
\label{ErrorL2} && \|u-u_H^h\|_{\Omega}\lesssim h^{r+1}+H^{r+2}.
\end{eqnarray}}
\end{corollary}

It is obvious that, to get the optimal $H^1$ or $L^2$ error, we should configure $H$ and $h$ such that
\begin{equation}\label{CHh}
h\sim H^\frac{r+1}{r}\quad\mbox{or}\quad h\sim H^\frac{r+2}{r+1},
\end{equation}
respectively. 

\section{Numerical Experiments}
In this section, we give some numerical examples to verify the analysis results. For simplicity, in all numerical examples we consider the following piecewise linear finite element spaces, that is $r=1$:
$$
S^H(\Omega)=\{v\in C^0(\Omega): v|_{\tau^H_\Omega}\in P^1_{\tau^H_\Omega},\forall \tau^H_\Omega\in T^H(\Omega)\},\quad S_0^H(\Omega)=S^H(\Omega)\cap H^1_0(\Omega).
$$

According to (\ref{CHh}), to reach the $H^1$ accuracy of the standard Galerkin approximation in $S_0^h(\Omega)$, we choose $H$ and $h$ such that $h\sim H^2$. In this case
\begin{equation}\label{5-1}
\|\nabla(u-u_H^h)\|_{\Omega}=O(H^2).
\end{equation}
On the other hand, to reach the $L^2$ accuracy of the standard Galerkin approximation in $S_0^h(\Omega)$, we choose $H$ and $h$ such that $h\sim H^\frac32$. With such configuration, we have
\begin{equation}\label{5-2}
\|u-u_H^h\|_{\Omega}=O(H^3).
\end{equation}
\begin{figure}[ht]
  \begin{minipage}[t]{0.5\linewidth}
    \centering
    \noindent\includegraphics[width=7cm]{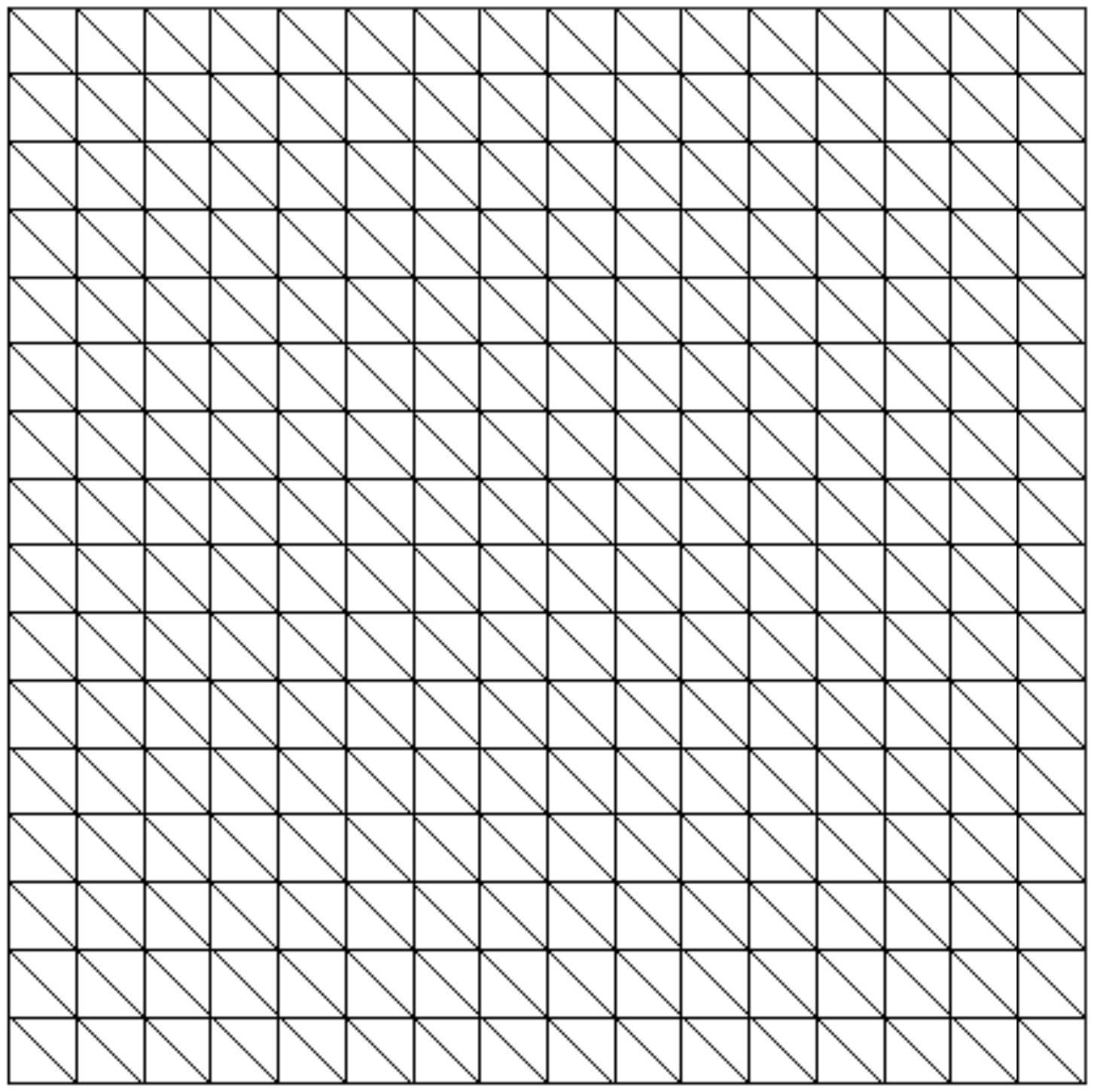}
    \caption{2-D coarse mesh $T^H(\Omega)$}
    \label{fig1}
  \end{minipage}%
  \begin{minipage}[t]{0.5\linewidth}
    \centering
   \noindent \includegraphics[width=7cm]{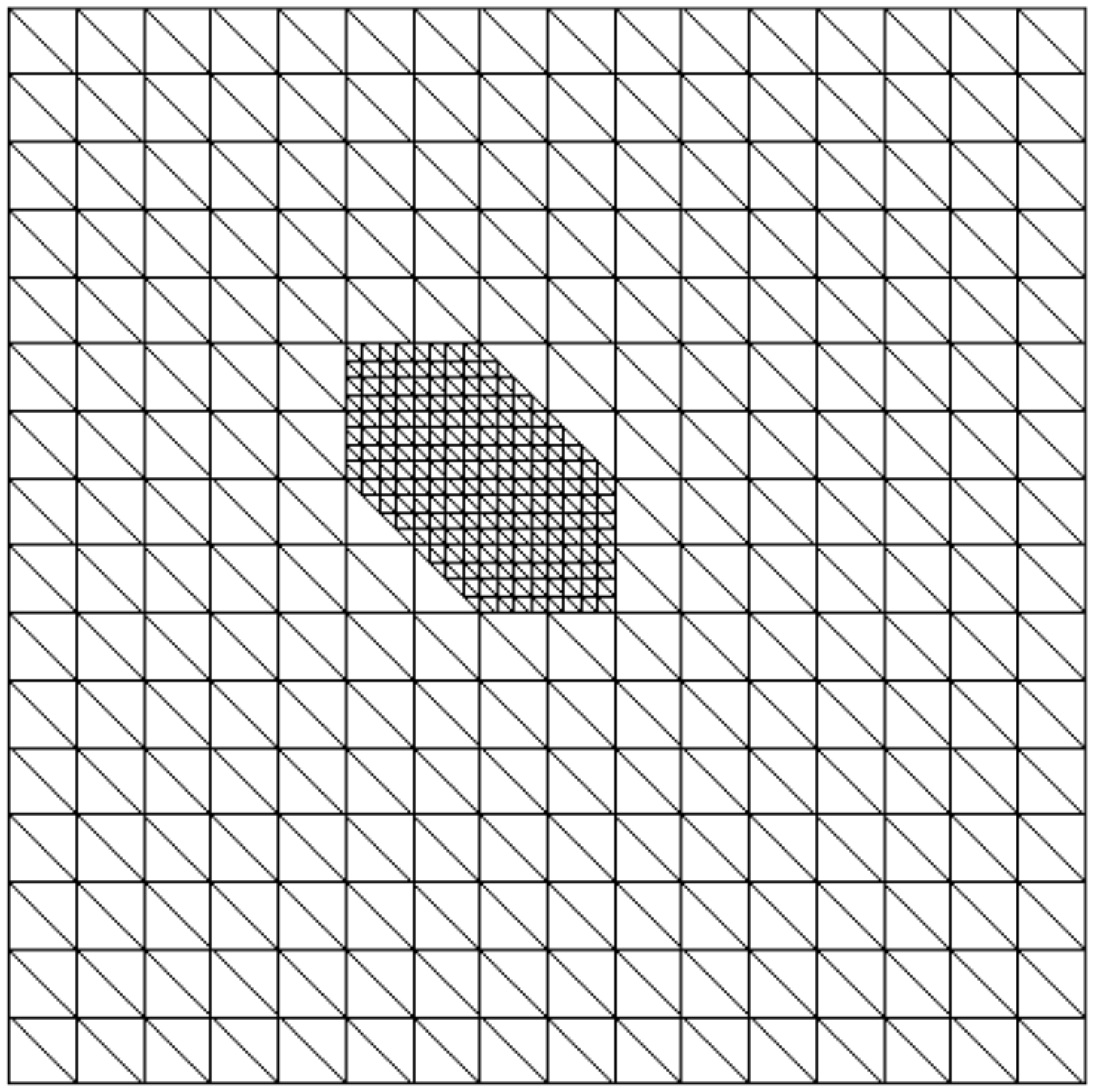}
    \caption{2-D local fine mesh $T^h(\Omega_j)$}
    \label{fig2}
  \end{minipage}
\end{figure}

In the following, we try to verify (\ref{5-1})-(\ref{5-2}) and the efficiency of the proposed iteration scheme (Step 0)-(Step 2) in last section by some numerical experiments. On the one hand, since all the subproblems in (Step 1) are independent with each other once the coarse mesh approximation $u_H$ is closed at hand, there will be no communication cost when solving them simultaneously in parallel computer systems. 
The only thing one should pay attention is the computing of the right hand side in (Step 1) and (Step 2). In (Step 1) we only need the coarse mesh information 
on $D_j$, the support of $\phi_j$. And in (Step 2), for calculating the second term of the right hand side of  (\ref{post2}), we calculate it in each fine mesh element, which dramatically increases the computing time compared with the standard coarse mesh Galerkin method. How to calculate this term efficiently is somehow critical to make the scheme more efficient. 
  
On the other hand, very large parallel computer systems that have huge amount of computing cores give us a possibility to deal with large scale computation and get very accurate approximation. In this case, the scale of the coarse mesh standard Galerkin scheme and the coarse mesh correction problem could be very large and how to solve the coarse mesh problem could be the bottleneck of the entire iterative scheme. Therefore, we should use some parallel solver to cope with the correction step, for example the algebraic multi-grid method which is well known by its efficiency.  In the following numerical experiment, we will compare the numerical performance of the proposed iterative scheme in last section with the fine grid standard Galerkin scheme. For the fine grid standard Galerkin scheme, we use pARMS (parallel Algebraic Multilevel Solver) as the parallel sparse solver.  In the following numerical experiments, the scale of both the fine mesh subproblems and the coarse mesh correction problem are not very large, so we only use direct method for solving them, although some parallel sparse solvers could be applied to the coarse mesh correction step when the scale of the coarse mesh correction is large. And it is shown that the efficiency of the proposed iteration scheme is higher than the fine mesh standard Galerkin scheme with the previously mentioned parallel sparse solver. 

First, we consider two 2-D examples. In these two examples, the domain is the unit square $\Omega=(0,1)\times (0,1)$ with a uniform triangulation $T^H(\Omega)=\{\tau^H_\Omega\}$, see Fig. \ref{fig1}. The Fig. \ref{fig2} is the local fine mesh defined on the $j$th coarse mesh node.

In the first 2-D example, we consider the problem with the following analytic solution
$$
u(x,y)=100(x^2-2x^3+x^4)(y-3y^2+2y^3).
$$
In this case, we can get the exact error of the numerical solution.

In the following Tab. \ref{tab1} and Tab. \ref{tab2}, we give some numerical results according to the above configurations of $H$ and $h$. And in Tab. \ref{tab1.1} and Tab. \ref{tab2.1} we give the CPU time comparison for the fine mesh standard Galerkin method and the scheme proposed in this paper with respect to getting optimal $H^1$ and $L^2$ errors, respectively. That is in Tab. \ref{tab1.1} we show the CPU time used when $h=H^2$ and in Tab. \ref{tab2.1} we show the CPU time used when $h=H^\frac32$. 
For numerical experiments that the true solution $u$ is known, we define the convergence order "$\mbox{ORDER}_1$" with respect to the coarse mesh size $H$ as
$$
\mbox{ORDER}_1(u_{app})=\left\{\begin{array}{ll}
1+\frac{\ln\frac{\|\nabla(u-u_H)\|_{0,\Omega}}{\|\nabla(u-u_{app})\|_{0,\Omega}}}{|\ln H|},\quad & H^1\;\mbox{error order},\\
2+\frac{\ln\frac{\|u-u_H\|_{0,\Omega}}{\|u-u_{app}\|_{0,\Omega}}}{|\ln H|},\quad & L^2\;\mbox{error order}.
\end{array}\right.
$$
The symbol $u_{app}$ stands for certain approximation of $u$ defined in the algorithm. And "Iteration" stands the number of iterations that are taken for obtaining the final approximation.

\begin{table}[ht]
\caption{$H^1-$error ($h=H^2$)}\label{tab1}
\begin{center}
\begin{tabular}{| c | c | c | c |}
\hline
$\frac{1}{H}$ & $8$ & $16$ & $32$\\
\hline
$\|\nabla(u-u_H)\|_{0,\Omega}$ & $6.8371\times 10^{-1}$ & $3.4949\times 10^{-1}$ & $1.7574\times 10^{-1}$\\
\hline
$\|\nabla(u-u_h)\|_{0,\Omega}$ & $8.7995\times 10^{-2}$ & $2.2009\times 10^{-2}$ & $5.5023\times 10^{-3}$\\
\hline
$\|\nabla(u-u_{H,h}^1)\|_{0,\Omega}$ & $1.6292\times 10^{-1}$ & $6.7335\times 10^{-2}$ & $2.3873\times 10^{-2}$\\
\hline
$\|\nabla(u-u_H^h)\|_{0,\Omega}$ & $8.8087\times 10^{-2}$  & $2.2041\times 10^{-2}$ & $5.5291\times 10^{-3}$\\
\hline
$\mbox{ORDER}_1$($u_{H,h}^1$) & $1.69$ & $1.59$ & $1.58$ \\
\hline
$\mbox{ORDER}_1$($u_H^h$) & $1.98$ & $1.9968$ & $1.9981$\\
\hline
Iteration & $2$ & $2$ & $2$\\
\hline
\end{tabular}
\end{center}
\end{table}

\begin{table}[ht]
\caption{CPU time comparison ($h=H^2$, $H^{-1}=32$)}\label{tab1.1}
\begin{center}
\begin{tabular}{| c | c | c | c |}
\hline
NP & $4$ & $8$ & $16$\\
\hline
CPU time ($u_h$) & $1532.71s$ & $940.56s$ & $559.12s$\\
\hline
CPU time ($u_H^h$) & $531.67s$ & $265.92s$ & $133.05$\\
\hline
\end{tabular}
\end{center}
\end{table}

\begin{table}[ht]
\caption{$L^2-$error ($h=H^\frac32$)}\label{tab2}
\begin{center}
\begin{tabular}{| c | c | c | c | c |}
\hline
$\frac{1}{H}$ & $25$ & $36$ & $49$ & $64$\\
\hline
$\|u-u_H\|_{0,\Omega}$ & $3.3784\times 10^{-3}$ & $1.6347\times 10^{-3}$ & $8.8361\times 10^{-4}$ & $5.1832\times 10^{-4}$\\
\hline
$\|u-u_h\|_{0,\Omega}$ & $1.3597\times 10^{-4}$ & $4.5545\times 10^{-5}$ & $1.8145\times 10^{-5}$ & $8.1066\times 10^{-6}$\\
\hline
$\|u-u_H^h\|_{0,\Omega}$ & $1.1498\times 10^{-4}$  & $3.8173\times 10^{-5}$ & $1.5086\times 10^{-5}$ & $7.5306\times 10^{-6}$\\
\hline
$\mbox{ORDER}_1(u_H^h)$ & $3.05$ & $3.05$ & $3.05$ & $3.02$\\
\hline
Iteration & $1$ & $1$ & $2$ & $2$\\
\hline
\end{tabular}
\end{center}
\end{table}

\begin{table}[ht]
\caption{CPU time comparison ($h=H^\frac32$, $H^{-1}=64$)}\label{tab2.1}
\begin{center}
\begin{tabular}{| c | c | c | c |}
\hline
NP & $4$ & $8$ & $16$\\
\hline
CPU time ($u_h$) & $86.56s$ & $63.17s$ & $49.47s$\\
\hline
CPU time ($u_H^h)$ & $119.34s$  & $64.76s$ & $36.12s$\\
\hline
\end{tabular}
\end{center}
\end{table}

The second 2-D example is defined by giving
$$
f=70\log ((x+0.1)(\sin\pi y+1)).
$$ 
In this example, since the exact solution $u$ is unknown, the convergence order of the approximate solution is calculated as 
$$
\mbox{ORDER}_2(u_{app})=\left\{\begin{array}{ll}
\min\{2,1+\frac{\ln\frac{\|\nabla(u_h-u_H)\|_{0,\Omega}}{\|\nabla(u_h-u_{app})\|_{0,\Omega}}}{|\ln H|}\}, & \quad H^1\;\mbox{error order},\\
\min\{3,2+\frac{\ln\frac{\|u_h-u_H\|_{0,\Omega}}{\|u_h-u_{app}\|_{0,\Omega}}}{|\ln H|}\}, & \quad L^2\;\mbox{error order}.
\end{array}\right.
$$
Here $u_h$ is the standard Galerkin approximation in the fine mesh finite element space $S_0^h(\Omega)$ and the symbol $u_{app}$ stands for certain approximation of $u$ defined in the algorithm. Since, for example, the $H^1$ error estimate of the fine mesh standard Galerkin approximation admits the following estimation when $h=H^2$ 
$$
\|\nabla(u-u_h)\|_{0,\Omega}=O(h)=O(H^2),
$$
the ``$\mbox{ORDER}_2(u_{app})$'' calculated by the above formula equals to 2 means
$$
\|\nabla(u_h-u_{app})\|_{0,\Omega}=O(H^2),
$$
therefore
$$
\|\nabla(u-u_{app})\|_{0,\Omega}=O(H^2).
$$

The Tab. \ref{tab3} and Tab. \ref{tab4} give the numerical results of this test problem. 

\begin{table}[ht]
\caption{$H^1-$error ($h=H^2$)}\label{tab3}
\begin{center}
\begin{tabular}{| c | c | c | c |}
\hline
$\frac{1}{H}$ & $8$ & $16$ & $32$\\
\hline
$\|\nabla(u_h-u_H)\|_{0,\Omega}$ & $1.8407\times 10^0$ & $9.7353\times 10^{-1}$ & $4.9548\times 10^{-1}$\\
\hline
$\|\nabla(u_h-u_{H,h}^1)\|_{0,\Omega}$ & $3.1398\times 10^{-1}$ & $1.6695\times 10^{-1}$ & $7.3657\times 10^{-2}$\\
\hline
$\|\nabla(u_h-u_H^h)\|_{0,\Omega}$ & $9.4723\times 10^{-2}$  & $4.1034\times 10^{-2}$ & $9.4765\times 10^{-4}$\\
\hline
$\mbox{ORDER}_2(u_{H,h}^1)$ & $1.85$ & $1.64$ & $1.55$ \\
\hline
$\mbox{ORDER}_2(u_H^h)$ & $2$ & $2$ & $2$\\
\hline
Iteration & $1$ & $1$ & $2$\\
\hline
\end{tabular}
\end{center}
\end{table}

\begin{table}[ht]
\caption{$L^2-$error ($h=H^\frac32$)}\label{tab4}
\begin{center}
\begin{tabular}{| c | c | c | c | c |}
\hline
$\frac{1}{H}$ & $25$ & $36$ & $49$ & $64$\\
\hline
$\|u_h-u_H\|_{0,\Omega}$ & $8.0991\times 10^{-3}$ & $3.9813\times 10^{-3}$ & $2.1717\times 10^{-3}$ & $1.2812\times 10^{-4}$\\
\hline
$\|u_h-u_H^h\|_{0,\Omega}$ & $2.6319\times 10^{-4}$  & $4.5438\times 10^{-6}$ & $1.9168\times 10^{-6}$ & $1.1412\times 10^{-6}$\\
\hline
$\mbox{ORDER}_2(u_H^h)$ & $3$ & $3$ & $3$ & $3$\\
\hline
Iteration & $1$ & $2$ & $2$ & $2$\\
\hline
\end{tabular}
\end{center}
\end{table}

In the rest of this section, we will give two 3-D numerical examples. In these two examples, the domain $\Omega$ is the unit cube $(0,1)^3$. The Fig. \ref{fig3} and \ref{fig4} give some free sketches of the 3-D coarse mesh and associated local fine mesh. 

The first 3-D example is a test problem with the following analytic solution
$$
u(x,y,z)=100(x^2-2x^3+x^4)(y-3y^2+2y^3)(z^3-z).
$$
The numerical results are given in Tab.\ref{tab5} and \ref{tab6}.

\begin{figure}[ht]
  \begin{minipage}[t]{0.5\linewidth}
    \centering
    \includegraphics[width=6cm]{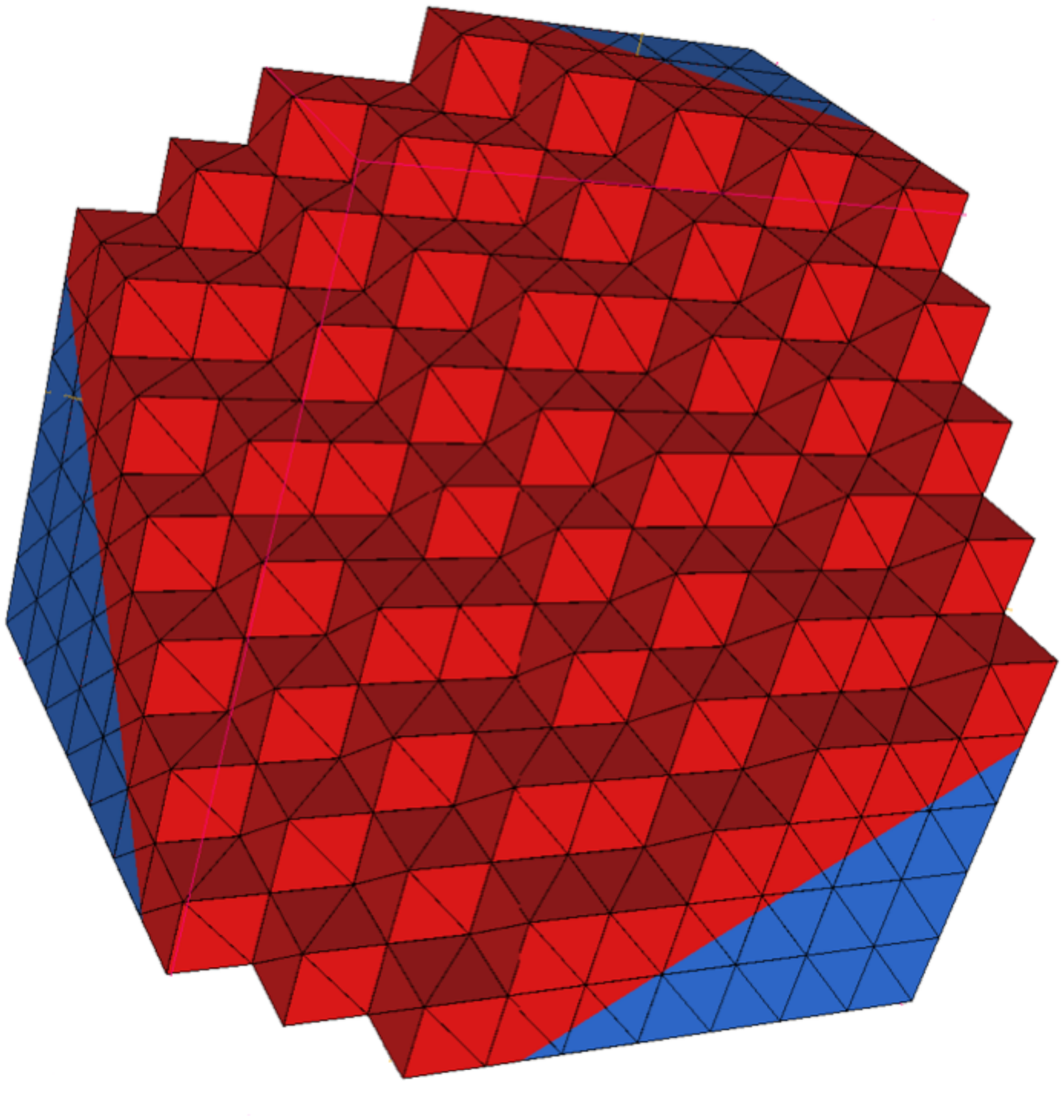}
    \caption{3-D coarse mesh $T^H(\Omega)$}
    \label{fig3}
  \end{minipage}%
  \begin{minipage}[t]{0.5\linewidth}
    \centering
    \includegraphics[width=6cm]{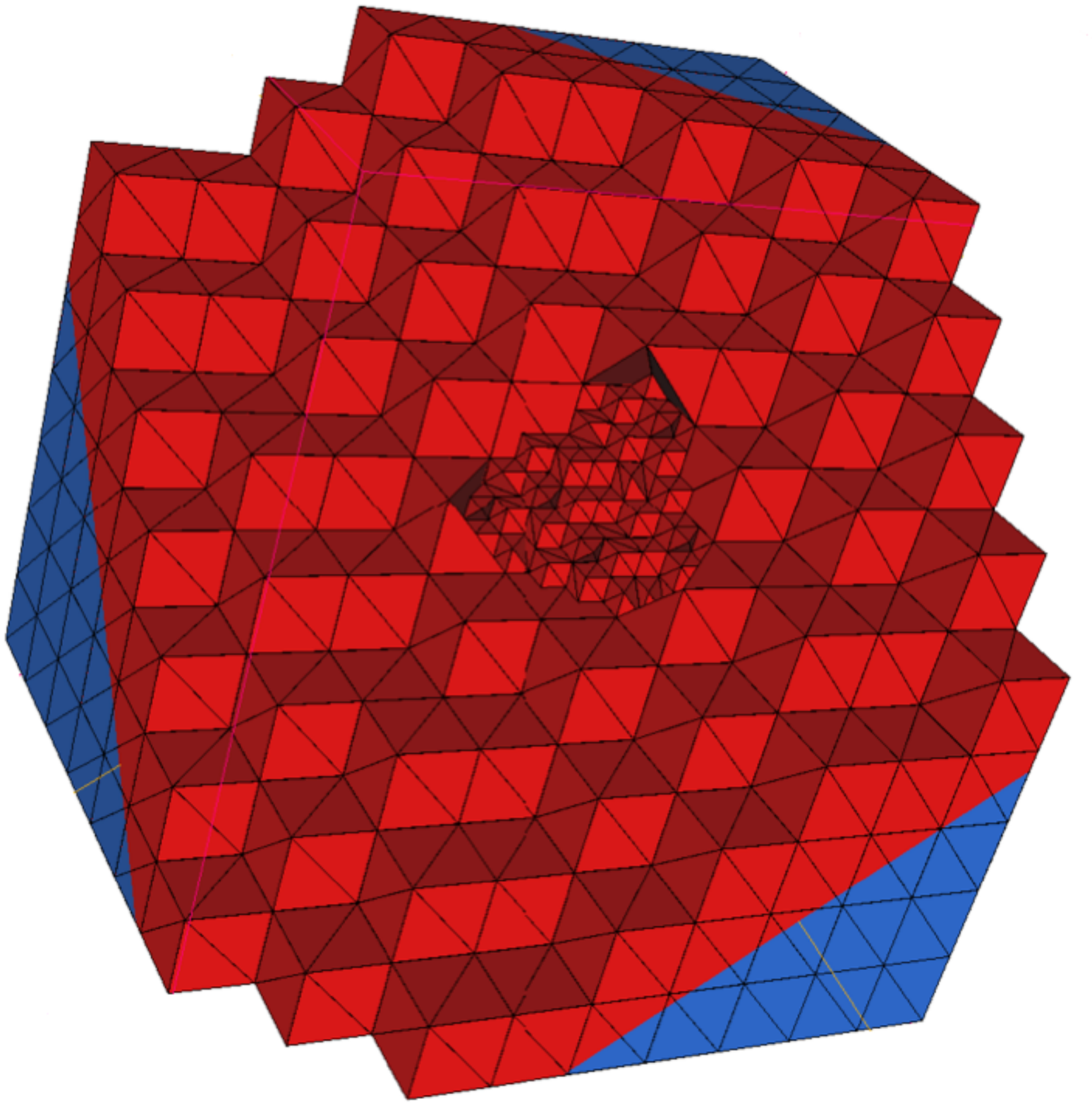}
    \caption{3-D local fine mesh $T^h(\Omega_j)$}
    \label{fig4}
  \end{minipage}
\end{figure}

\begin{table}[ht]
\caption{$H^1-$error ($h=H^2$)}\label{tab5}
\begin{center}
\begin{tabular}{| c | c | c | c |c|}
\hline
$\frac{1}{H}$ & $6$ & $8$ & $10$ & $12$\\
\hline
$\|\nabla(u-u_H)\|_{0,\Omega}$ & $2.9515\times 10^{-1}$ & $2.2774\times 10^{-1}$ & $1.8467\times 10^{-1}$ & $1.5504\times 10^{-1}$\\
\hline
$\|\nabla(u-u_h)\|_{0,\Omega}$ & $5.9048\times 10^{-2}$ & $3.3496\times 10^{-2}$ & $2.1515\times 10^{-2}$ & $1.4969\times 10^{-2}$\\
\hline
$\|\nabla(u-u_{H,h}^1)\|_{0,\Omega}$ & $8.5886\times 10^{-2}$ & $6.2379\times 10^{-2}$ & $4.7159\times 10^{-2}$ & $3.6729\times 10^{-2}$\\
\hline
$\|\nabla(u-u_H^h)\|_{0,\Omega}$ & $5.9436\times 10^{-2}$  & $3.3868\times 10^{-2}$ & $2.1682\times 10^{-2}$ & $1.5110\times 10^{-2}$\\
\hline
$\mbox{ORDER}_1(u_{H,h}^1)$ & $1.69$ & $1.62$ & $1.59$ & $1.58$\\
\hline
$\mbox{ORDER}_1(u_H^h)$ & $1.89$ & $1.92$ & $1.93$ & $1.94$\\
\hline
Iteration & $2$ & $2$ & $2$ & $2$\\
\hline
\end{tabular}
\end{center}
\end{table}

\begin{table}[ht]
\caption{$L^2-$error ($h=H^\frac32$)}\label{tab6}
\begin{center}
\begin{tabular}{| c | c | c | c |}
\hline
$\frac{1}{H}$ & $9$ & $16$ & $25$ \\
\hline
$\|u-u_H\|_{0,\Omega}$ & $8.6338\times 10^{-3}$ & $2.8568\times 10^{-3}$ & $1.1850\times 10^{-3}$ \\
\hline
$\|u-u_h\|_{0,\Omega}$ & $1.2840\times 10^{-3}$ & $2.3929\times 10^{-4}$ & $6.3981\times 10^{-5}$ \\
\hline
$\|u-u_H^h\|_{0,\Omega}$ & $1.3024\times 10^{-3}$  & $2.3795\times 10^{-4}$ & $6.2568\times 10^{-5}$ \\
\hline
$\mbox{ORDER}_1(u_H^h)$ & $2.86$ & $2.90$ & $2.91$ \\
\hline
Iteration & $2$ & $2$ & $2$ \\
\hline
\end{tabular}
\end{center}
\end{table}

The second example of 3-D case is a test problem driven by the following free term
$$
f=70\log((x+0.1)(\sin\pi y+1)(z+0.1)(\sin\pi z+1)),
$$
whose numerical results are given in Tab.\ref{tab7} and \ref{tab8}.

\begin{table}[ht]
\caption{$H^1-$error ($h=H^2$)}\label{tab7}
\begin{center}
\begin{tabular}{| c | c | c | c |c|}
\hline
$\frac{1}{H}$ & $6$ & $8$ & $10$ & $12$\\
\hline
$\|\nabla(u_h-u_H)\|_{0,\Omega}$ & $3.1859\times 10^{0}$ & $2.5719\times 10^{0}$ & $2.1384\times 10^{0}$ & $1.8236\times 10^{0}$\\
\hline
$\|\nabla(u_h-u_{H,h}^1)\|_{0,\Omega}$ & $5.9544\times 10^{-1}$ & $5.1948\times 10^{-1}$ & $4.3693\times 10^{-1}$ & $3.6855\times 10^{-1}$\\
\hline
$\|\nabla(u_h-u_H^h)\|_{0,\Omega}$ & $5.5748\times 10^{-2}$  & $4.0252\times 10^{-2}$ & $2.9719\times 10^{-2}$ & $2.2962\times 10^{-2}$\\
\hline
$\mbox{ORDER}_2(u_{H,h}^1)$ & $1.94$ & $1.77$ & $1.69$ & $1.64$\\
\hline
$\mbox{ORDER}_2(u_H^h)$ & $2$ & $2$ & $2$ & $2$\\
\hline
Iteration & $2$ & $2$ & $2$ & $2$\\
\hline
\end{tabular}
\end{center}
\end{table}

\begin{table}[ht]
\caption{$L^2-$error ($h=H^\frac32$)}\label{tab8}
\begin{center}
\begin{tabular}{| c | c | c | c |}
\hline
$\frac{1}{H}$ & $9$ & $16$ & $25$\\
\hline
$\|u_h-u_H\|_{0,\Omega}$ & $7.7114\times 10^{-2}$ & $2.8568\times 10^{-2}$ & $1.2454\times 10^{-2}$\\
\hline
$\|u_h-u_H^h\|_{0,\Omega}$ & $5.7182\times 10^{-3}$  & $1.5811\times 10^{-3}$ & $4.2787\times 10^{-5}$\\
\hline
$\mbox{ORDER}_2(u_H^h)$ & $3$ & $3$ & $3$\\
\hline
Iteration & $1$ & $1$ & $2$\\
\hline
\end{tabular}
\end{center}
\end{table}

All the above numerical results are obtained by using the public domain software FreeFem++ \cite{freefem++}.

{\bf Remark} {\em From the construction of the partition of unity used in the algorithm, we see that the computational domain of each local subproblem is contained in a ball with radius of $O(H)$. This means that the volume of the computational domain of each local subproblem tends to zero as the coarse mesh size $H$ tends to zero. In this sense, we call the algorithm given in this paper an expandable local and parallel two-grid finite element algorithm. }


\section*{Acknowledgment}
This work is supported by NSFC (Grant No. 11571274 \& 11171269) and the Ph.D. Programs Foundation of Ministry of Education of China (Grant No. 20110201110027). And we would like to thank the anonymous referees for their valuable suggestions and comments. 

\end{document}